\documentclass[a4,11pt,english]{amsart}

\usepackage[misc]{ifsym}
\usepackage{latexsym,paralist}
\usepackage{amsmath}
\usepackage{amssymb}
\usepackage{amsthm}
\usepackage[mathcal]{eucal}
\usepackage{amsmath, amssymb,enumerate}
\usepackage{mathrsfs}
\usepackage{enumerate}
\usepackage[colorlinks=true, pdfstartview=FitV, linkcolor=blue, citecolor=blue,
 urlcolor=blue]{hyperref} 
\usepackage{color}
\usepackage{comment}
\allowdisplaybreaks

\usepackage[pagewise]{lineno}%\linenumbers

\makeatletter \@addtoreset{equation}{section}

\makeatother \makeatletter

\newtheorem{theorem}{Theorem}
\newtheorem{proposition}{Proposition}
\newtheorem{lemma}{Lemma}
\newtheorem{remark}{Remark}

\newtheorem{coro}{Corollary}

\newcommand{\cc}{C_{c}^{\infty}(\R^{N})}

\newcommand{\R}{\mathbb{R}}

\newcommand{\C}{\mathbb{C}}

\newcommand{\N}{\mathbb{N}}

\newcommand{\Rp}{\textrm{\emph{Re}\,}}

\allowdisplaybreaks
\begin{document}
\title[]{$L^p$-regularity for fourth-order Schr\"odinger operators with unbounded or degenerate coefficients}

\author{Federica Gregorio} 
\address{Dipartimento di Matematica, Università degli Studi di Salerno, Fisciano, Italy}
\email{fgregorio@unisa.it}

\author{Chiara Spina}
\address{Dipartimento di Matematica e Fisica "Ennio De Giorgi", Università del Salento, Lecce, Italy}
\email{chiara.spina@unisalento.it} 

\author{Cristian Tacelli}
\address{Dipartimento di Matematica, Università degli Studi di Salerno, Fisciano, Italy}
\email{ctacelli@unisa.it}

\thanks{The authors are members of the Gruppo Nazionale per l’Analisi Matematica, la Probabilità
e le loro Applicazioni (GNAMPA) of the Istituto Nazionale di Alta Matematica (INdAM)}
%\thanks{$^*$Corresponding author: }


\begin{abstract}
We prove that  operators of the form $A=-\Delta^{2}-V(x)$, with $|D V(x)|\leq \gamma V(x)^\frac{5}{4}+ C_\gamma$ for some $\gamma$ small enough and $C_\gamma$ positive constant,  generate  analytic semigroups in $L^p(\R^N)$ for $1\leq p<\infty$ and in $C_0(\R^N)$.  Moreover, we characterize the maximal domain of such operators.
\end{abstract}

\maketitle

%\centerline{\scshape
%Federica Gregorio$^{%{\href{mailto:fgregorio@unisa.it}%{\textrm{\Letter}}}
%*1}$, Chiara Spina$^{2}$, Cristian Tacelli$^{1}$}

\medskip

{\footnotesize
% Enter the full affiliation and country name:
% Do not insert commas or periods at the end of lines.
 %\centerline{$^1$Dipartimento di Matematica, Università degli Studi di Salerno, Fisciano, Italy}
} % Do not forget to end {\footnotesize with the sign }

\medskip

{\footnotesize
 % Enter the full affiliation and country name:
 %\centerline{$^2$Dipartimento di Matematica e Fisica "Ennio De Giorgi", Università del Salento, Lecce, Italy}
}

\section{Introduction}

In the literature, numerous results concern second-order elliptic operators with unbounded or degenerate coefficients. In particular, elliptic and parabolic solvability in $L^p$ spaces and domain characterizations have been widely studied. By contrast, higher-order operators with unbounded coefficients remain less investigated. Nevertheless, driven by their physical relevance, there has been growing interest in the higher-order setting in recent years. For instance, results have been obtained for the bi-Ornstein-Uhlenbeck operator  in $L^2_\mu(\mathbb{R}^N)$ with the invariant measure $\mu$ (see \cite{agrt}), as well as for operators with unbounded coefficients in the fourth-order part, such as $(1+\vert{}x\vert{}^2)^\alpha \Delta^2$ for $0 \leq \alpha \leq 2$ in $L^p$ spaces (see \cite{gre-tac,gre-spi-tac1,gre-spi-tac2}). Furthermore, fourth-order Schrödinger operators $\Delta^2 + V$ have been examined for potentials $V$ in suitable higher-order Kato classes or satisfying bounds on their derivatives up to the third order (see \cite{dav-hin98b,gre-spi-tac3}).

\medskip
In this paper we focus on  Schr\"odinger operators  with positive and possibly unbounded potentials $V$. We are mainly interested in finding conditions on the potential  which imply elliptic and parabolic solvability in $L^p$ spaces. Moreover, we aim to describe the domain of the operator precisely.  
Using semigroup theory, we address solvability questions by proving the generation of analytic semigroups.

\medskip
The analogous problem for second-order operators has been completely solved. Indeed, generation results were first established in the pioneering work by Kato (\cite{kato86}), while the characterization of the domain, even for more general operators, was proved in \cite{oka84,met-pru-rha-sch,aus-ben,she,can-ves}. In view of our results, we refer particularly to those papers where the authors proved that, under an oscillation condition on the potential $V$ of the form $\vert{}DV\vert{} \leq \gamma V^{3/2} + C_\gamma$ for sufficiently small $\gamma$, the operator $\Delta - V$ generates an analytic semigroup in $L^p$ spaces or in spaces of continuous functions, with its domain given by the intersection of the domain of the Laplacian and that of the potential viewed as a multiplication operator.

It is important to note that dissipativity plays a crucial role in proving such results in the second-order case. Since dissipativity and other classical second-order properties, such as domination or maximum principles, do not generally hold in the fourth-order setting, the second-order proofs cannot simply be adapted, and a different approach is required. 

\medskip
Under the assumption $\vert{}DV\vert{} \leq \gamma V^{5/4} + C_\gamma$ for sufficiently small $\gamma$, we prove that the operator $-\Delta^2 + V$ generates an analytic semigroup in $L^p$ spaces and in spaces of continuous functions. Moreover, its domain is given by the intersection of the domain of the bi-Laplacian and that of the potential viewed as a multiplication operator. In fact, we can weaken the regularity assumption on $V$ by requiring it to be bounded from above and below by a smooth function satisfying the aforementioned condition.

\medskip
Let us briefly outline the steps toward the main result and the structure of the paper. We begin by recalling some interpolative inequalities and covering results with balls of variable radius. Next, we observe that our assumption on the gradient of $V$ allows us to apply these covering results to balls whose radii are given by a suitable power of the potential. From this, we deduce weighted interpolative inequalities where the weight is a power of the potential. Equipped with these preliminary tools, by using the well-known sectoriality of the bi-Laplacian operator alongside the covering results, we prove a priori estimates in a space of smooth and compactly supported functions. Finally, we extend these estimates to the maximal domain of the operator and characterize it.

In the case $1 < p < \infty$, inspired by classical Agmon's method, as a next step we introduce an auxiliary operator involving an additional variable and deduce a priori estimates for it by arguing as before. Applying these estimates to a suitable function allows us to establish an elliptic resolvent estimate for $A$. Next, by approximation and compactness arguments, we establish elliptic solvability for the operator. Consequently, we obtain the sectoriality property and, via classical generation results, prove that $A$ generates an analytic semigroup in $L^p(\mathbb{R}^N)$ for $1 < p < \infty$.

 The cases $p=1$ and $p=\infty$ need to be treated separately. By using an argument due to Stewart and Masuda, based on Morrey's emmbedding, we prove that $A$ generates an analytic semigroup also in the space $C_0(\R^N)$ of continuous functions vanishing at infinity. The a priori estimates guarantee 
 the domain's description also in this case.
  
  Finally, by using a duality argument, we deduce analogous results also in the $L^1$ setting.

 \medskip
 We point out that our results apply to potentials  behaving like $|x|^r \,(r > 0)$ at infinity, to certain exponential-type potentials, and, in the degenerate case, to potentials like $ |x|^{-s}$ for $s > 4$.

\medskip
 \textbf{Notations.} We use standard notations for function spaces. We denote by $L^p(\R^N)$ and $W^{k,p}(\R^N)$ the standard Lebesgue and Sobolev spaces, respectively. $C_c^\infty(\R^N)$ is the space of test functions, $C_0(\R^N)$ is the space of  continuous functions going to $0$ when $|x|$ goes to infinity and $C_b(\R^N)$ is the space of bounded continuous functions. We denote by $W_{\rm comp}^{4,p}(\R^{N})$   the space of functions in $W^{4,p}(\R^N)$ with compact support.
$B_R$ denotes the open ball of radius $R$ centered at 0, whereas $B_R(x_0)$ is the ball centered at $x_0$ for some $x_0\in\R^N$.  We denote by $B_R^c$ the complementary set of the open ball, $B_R^c=\R^N\setminus B_R$. Moreover, we use the symbol $\|\cdot\|_{p,R}$ to shorten the notation $\|\cdot\|_{L^p(B_R(x_0))}$. In the proofs, $C$ is a  positive constant that may vary from line to line. Finally, for every function $u$ smooth enough we set
\begin{align*}|D^4u|=\left(\sum_{i,j,k,l=1}^N|D_{ijkl}u|^2\right)^\frac12&, \ |D^3u|=\left(\sum_{i,j,k=1}^N|D_{ijk}u|^2\right)^\frac12\\|D^2u|=\left(\sum_{i,j=1}^N|D_{ij}u|^2\right)^\frac12&, \ |Du|=\left(\sum_{i=1}^N|D_{i}u|^2\right)^\frac12.\end{align*}

\section{A priori estimates and domain characterization}
Let us consider operators of the form  $A=-\Delta^2-V(x)$ with $V:\R^N\to\R$ a  measurable function satisfying
\begin{equation} \label{gradientC}\tag{$H$} 
 \frac{1}{c_0}  U\leq V\leq c_0U\text{ with }
U \in C^1(\R^N),\
 U(x)>0\ \forall\,x\in\R^N,\ |D U(x)|\leq \gamma U(x)^\frac{5}{4}+ C_\gamma,
\end{equation}
for some positive and small enough constant $\gamma$ and some constant $C_\gamma \ge 0$,  $c_0>1$ near enough to $1$ as will be specified later.
Observe that, if $C_\gamma=0$, the  last inequality    implies $|D U ^{-\frac{1}{4}}|\leq \frac{\gamma}{4}$. We will actually prove all results under the strongest assumption 
\begin{equation} \label{gradient}\tag{$H'$} 
 \frac{1}{c_0}  U\leq V\leq c_0U\text{ with }
U \in C^1(\R^N),\
 U(x)>0\ \forall\,x\in\R^N,\ |D U(x)|\leq \gamma U(x)^\frac{5}{4},
\end{equation}
for some positive and small enough constant $\gamma$ and   $c_0>1$ near enough to $1$. Once we have generation of an analytic semigroup under assumption \eqref{gradient}, the result will also hold for potentials satisfying \eqref{gradientC} by translation.

\subsection{Some preliminary results}
We show some preliminary interpolative inequalities for the  biharmonic operator.

\begin{lemma}  \label{interp-bila}
Let $u\in C_c^\infty(\R^N)$ and $1\leq p\leq \infty$. Then, for every $\varepsilon>0$ we have
 \begin{align}\label{eq:inter0}
 \|Du\|_p
	&\leq C\|\Delta u\|_p^\frac{1}{2}\|u\|^\frac{1}{2}_p\leq \varepsilon \|\Delta u\|_{p}+\frac{C}{\varepsilon}\|u\|_{p}
\end{align}
and
\begin{align} \label{eq:inter1}
\|D^3u\|_p
	&\leq C\|\Delta^2 u\|_p^\frac{3}{4}\|u\|^\frac{1}{4}_p\leq \varepsilon \|\Delta^2 u\|_{p}+\frac{C}{\varepsilon^{3}}\|u\|_{p}
\end{align}
%\|D^2u\|_p
	%&\leq C\|\Delta^2 u\|_p^\frac{1}{2}\|u\|^\frac{1}{2}_p\leq %\varepsilon^2 \|\Delta^2 u\|_{p}+\frac{C}%{\varepsilon^{2}}\|u\|_{p}\label{eq:inter2}\\
%\|Du\|_p
	%&\leq C\|\Delta^2u\|^\frac{1}{4}_p\|u\|^\frac{3}{4}_p\leq %\varepsilon^3 \|\Delta^2 u\|_{p}+\frac{C}{\varepsilon}\|u\|_{p}. %\label{eq:inter3}
\end{lemma}
\begin{proof}
In the case $N=1$ the inequalities \eqref{eq:inter0} and \eqref{eq:inter1} follow by straightforward computations using Taylor's formula with integral remainder. 
Let now $N>1$ and  $\lambda>0$. We  prove (\ref{eq:inter1}).  Set $\lambda u+\Delta^2u=f$.  By setting $m=4$ in \cite[Theorem 5.7]{tan}, the fundamental solution $K_\lambda(x,y)$ of the operator $\lambda+\Delta^2$ satisfies
$$|D_x^3 K_\lambda (x,y)|\leq Ce^{-|\lambda|^\frac{1}{4}|x-y|}|x-y|^{1-N}.
%|f(y)|dy.
$$
By Young's inequality, set $G_\lambda (z)=e^{-|\lambda|^\frac{1}{4}|z|}|z|^{1-N}$, $z\in\R^N$, it  follows that 
$$\|D^3u\|_p\leq \|G_\lambda\|_1\|f\|_p\leq C|\lambda|^{-\frac{1}{4}}\|f\|_p=|\lambda|^{-\frac{1}{4}}\|\lambda u+\Delta^2 u\|_p\leq C|\lambda|^\frac{3}{4}\|u\|_p+|\lambda|^{-\frac{1}{4}}\|\Delta^2 u\|_p.$$
By choosing $\lambda=\|\Delta^2 u\|_p\|u\|_p^{-1}$ inequality \eqref{eq:inter1} follows. The proof of \eqref{eq:inter0} follows in a similar way.
\end{proof} 

In order to prove some weighted interpolation estimates and to prove a priori estimates for our operator we need a covering result with variable radius balls. The following formulation can be found in  in \cite{cup-for}.
\begin{proposition}\cite[Proposition 6.1]{cup-for} \label{pr:covering}
Let ${\mathcal F} =\{  B_{\rho (x)}(x)\}_{x\in \R^{N}}$, 
where $\rho : \R^{N} \to \R_{+}$ is a Lipschitz continuous function with Lipschitz
constant $\kappa$ strictly less than $\frac 12$.
Then, there exist a countable subcovering $\{ B_{\rho(x_{n})}(x_{n})\}_{n\in \N}$
 and a natural number $\xi=\xi (N,\kappa)$ such that at most $\xi$ 
among the doubled balls $\{ B_{2\rho(x_{n})}(x_{n} )\}_{n\in \N}$ overlap.
\end{proposition}

Observe that by the inequality in (\ref{gradient}) satisfied by  $U(x)$, the function $\rho(x)=\frac{2\eta}{\gamma }U(x)^{-\frac{1}{4}}$, $0<\eta\leq 1$, satisfies the assumptions of the previous covering proposition with Lipschitz constant smaller than $\frac{\eta}{2}$. 

The next lemma constitutes a first step toward the a priori estimates.

\begin{lemma}  \label{interpolation}
Let  $1\leq p\leq \infty$ and assume that $V$ satisfies (\ref{gradient}). Then there exists a positive constant  $C=C(N,p,c)$ such that for every $\varepsilon>0$ the following inequalities hold
\begin{align}
&\|V^{\frac{3}{4}}Du\|_{p}\leq \varepsilon ^{3}\|\Delta^2u\|_{p}+\frac{C}{\varepsilon}\|Vu\|_{p} \label{eq:interp00-01}\\
&\|V^\frac{1}{2}D^{2}u\|_{p}\leq \varepsilon^2 \|\Delta^2 u\|_{p}+\frac{C}{\varepsilon^2}\|Vu\|_{p}\label{eq:interp00-02} \\
&\|V^{\frac{1}{4}}D^{3}u\|_{p}\leq \varepsilon \|\Delta^2 u\|_{p}+\frac{C}{\varepsilon^{3}}\|Vu\|_{p} \label{eq:interp00-03}
\end{align}
for every $u\in C_{c}^{\infty}(\R^{N})$. 
\end{lemma}
\begin{proof}  
First  we prove that for $h=1,2,3$ and for $v\in C_{c}^{\infty}(\R^{N})$
\begin{equation}\label{eq:interp0}
\|V^{\frac{h}{4}}D v\|_{p}\leq C \|V^{\frac{h-1}{4}}\Delta v\|^{\frac 12}_{p}\|V^{\frac{h+1}{4}}v\|^{\frac 12}_{p}.
\end{equation}
%from which \eqref{eq:interp1} easily follows replacing $v$ by     %$D^{3-h}u$.
Let us set $\rho(x)=\frac{1}{\gamma}U(x)^{-\frac{1}{4}}$ and note that, as stated before,  by \eqref{gradient} it follows that $\rho(x)$ satisfies the assumptions of Proposition \ref{pr:covering} with Lipschitz constant smaller than $\frac{1}{4}$. We fix  $x_{0}\in \R^{N}$ and set $R=\rho (x_{0})=\frac{1}{\gamma}U(x_0)^{-\frac{1}{4}}$. 
Then, if $x\in B_{2R}(x_0)$, 
since $|DU^{-\frac{1}{4} }|\leq \frac{\gamma}{4}$
we have $\left |U^{-\frac{1}{4} }(x)-U^{-\frac{1}{4} }(x_0)\right|\leq 2R\frac{\gamma}{4}=\frac{1}{2}U^{-\frac{1}{4} }(x_0)$ and then 
$\frac{1}{2}U^{-\frac{1}{4} }(x_0)\leq U^{-\frac{1}{4} }(x)\leq \frac{3}{2}U^{-\frac{1}{4} }(x_0)$, taking into account the hypothesis on $V$, the following inequalities hold
 \[
 c_1V(x_{0})^{\frac{1}{4}}\leq V(x)^{\frac{1}{4}}\leq c_2V(x_{0})^{\frac{1}{4}}
\] where $c_1=\frac {2}{3}c_0^{-\frac 12}$ and $c_2= 2 c^{\frac 12}_0$ and consequently 
\[
 c_1^{ k}V(x_{0})^{\frac{k}{4}}\leq V(x)^{\frac{k}{4}}\leq c_2^{k}V(x_{0})^{\frac{k}{4}}
\]
for $k=h-1,h,h+1$. 
Let now  $\vartheta\in C_c^{\infty }(\R^N)$ be such that $0\leq \vartheta \leq 1$,
$\vartheta(x)=1$ for $x\in B_{R}(x_{0}),\vartheta(x)=0$ for $x\in B^{c}_{2R}(x_{0})$ and $|D^h\vartheta|\leq {C}{R^{-h}}$ for $h=1,2,3$.
By Lemma \ref{interp-bila}, we have
\begin{align*}
\|V^{\frac{h}{4}}D v\|_{p, R}
&\leq C\|V(x_{0})^{\frac{h}{4}}D v\|_{p, R}\leq  C\|V(x_{0})^{\frac{h}{4}}D (\vartheta v)\|_{p}\\
& \leq C  \| V(x_{0})^{\frac{h-1}{4}} \Delta (\vartheta v)\|_{p}^{\frac{1}{2}}  \| V(x_{0})^{\frac{h+1}{4}} \vartheta v\|^{\frac{1}{2}}_{p}\\
& \leq \varepsilon  \|V(x_{0})^{\frac{h-1}{4}} \Delta (\vartheta v)\|_{p}
		+\frac{C}{\varepsilon } \| V(x_{0})^{\frac{h+1}{4}} \vartheta v\|_{p}\\
&  \leq \varepsilon  \|V(x_{0})^{\frac{h-1}{4}} \Delta v\|_{p,2R}+\frac{C}{R}\varepsilon \|V(x_{0})^{\frac{h-1}{4}} D v\|_{p,2R}
		+\frac{C}{R^{2}}\varepsilon \|V(x_{0})^{\frac{h-1}{4}} v\|_{p,2R}\\
&\quad		+\frac{C}{\varepsilon}\| V(x_{0})^{\frac{h+1}{4}} v\|_{p,2R}.
\end{align*}
Then,
taking into account that 
$\frac{1}{R}V(x_{0})^{\frac{h-1}{4}}=c V(x_{0})^{\frac{h}{4}}\leq CV(x)^{\frac{h}{4}}$
and $\frac{1}{R^{2}}V(x_{0})^{\frac{h-1}{4}}=c V(x_{0})^{\frac{h+1}{4}}\leq CV(x)^{\frac{h+1}{4}}$ in $B_{2R}(x_0)$ one has
\begin{align}\label{eq:cover-x00}
 \|V^{\frac{h}{4}}D v\|_{p,R} \leq C \left( \varepsilon \|V^{\frac{h-1}{4}}\Delta v\|_{p,2R}+\varepsilon \|V^{\frac{h}{4}}D v\|_{p,2R}
		+\left(\varepsilon+\frac{1}{\varepsilon}\right)\|V^{\frac{h+1}{4}}v\|_{p,2R}
		\right).
\end{align}
Assume first $1\leq p<\infty$. Let $\{B_{\rho(x_n )}(x_n)\}$ be a countable covering of $\R^N$ as in Proposition \ref{pr:covering}
such that at most $\zeta $ among the double balls $\{B_{2\rho(x_n )}(x_n)\}$ overlap.
We write \eqref{eq:cover-x00} with $x_0$ replaced by $x_n$ and sum over $n$.
Taking into account the above covering result, we get
$$
\|V^{\frac{h}{4}}D v\|_{p}\leq 
C \left( \varepsilon \|V^{\frac{h-1}{4}}\Delta v\|_{p}+\varepsilon \|V^{\frac{h}{4}}D v\|_{p}
		+\left(1+\frac{1}{\varepsilon}\right)\|V^{\frac{h+1}{4}}v\|_{p}\nonumber
		\right).
$$
When $p=\infty$, the proof follows immediately without using the covering result.
Choosing $\varepsilon'=\dfrac{C\varepsilon}{1-C\varepsilon}$   one can find a suitable positive $C$ such that  
$$
\|V^{\frac{h}{4}}D v\|_{p}\leq 
\varepsilon' \|V^{\frac{h-1}{4}}\Delta v\|_{p}
+\dfrac{C}{\varepsilon'}\|V^{\frac{h+1}{4}}v\|_{p}\nonumber
$$
from which 
\eqref{eq:interp0} follows.

Applying \eqref{eq:interp0} to $v$, $D_{j}v$ and $\Delta v$   we can write
\begin{align}
    \|V^\frac{3}{4}D_iv\|_p&\leq C\|V^\frac{1}{2}\Delta v\|_p^\frac{1}{2}\|Vv\|_p^\frac{1}{2}
            \label{eq:interp01-01}\\
    \|V^\frac{1}{2}D_{ij}v\|_p&\leq C\|V^\frac{1}{4}\Delta D_jv\|_p^\frac{1}{2}\|V^\frac{3}{4}D_jv\|_p^\frac{1}{2}
            \label{eq:interp01-02}\\
    \|V^\frac{1}{4}D_i\Delta v\|_p&\leq C\|\Delta^2 v\|_p^\frac{1}{2}\|V^\frac{1}{2}\Delta v\|_p^\frac{1}{2}.
            \label{eq:interp01-03}
\end{align}

Combining \eqref{eq:interp01-02} with \eqref{eq:interp01-01} and \eqref{eq:interp01-03} we obtain 
$\|V^\frac{1}{2}D_{ij}v\|_p\leq C 
    \|\Delta^2 v\|_p^\frac{1}{4}
    \|V^\frac{1}{2}\Delta v\|_p^\frac{1}{2}
    \|Vu\|_p^\frac{1}{4} $ 
from which we have
\begin{equation}\label{eq:interp01-04} 
\|V^\frac{1}{2}D^2v\|_p\leq C 
    \|\Delta^2 v\|_p^\frac{1}{2}
    \|Vu\|_p^\frac{1}{2}.     
\end{equation}
Then, \eqref{eq:interp00-02} follows.

Combining \eqref{eq:interp01-01} with \eqref{eq:interp01-04} we obtain
\[
\|V^\frac{3}{4}D_iv\|_p\leq C 
    \|\Delta ^2v\|_p^\frac{1}{4}
    \|Vv\|_p^\frac{3}{4}
\]
from which, by Young's inequality, \eqref{eq:interp00-01} follows.

Now, by arguing in similar way and starting from the interpolation inequality
\eqref{eq:inter1}, we obtain 
\eqref{eq:interp00-03}.
\end{proof}

We introduce the following spaces
\begin{align*}
D(A_1)&=\{u\in L^1(\R^N):\ Vu,\   \Delta^2 u\in L^1(\R^N)\}\\
D(A_p)&=\{u\in W^{4,p}(\R^N):\ Vu,\   D^4 u\in L^p(\R^N)\},\quad 1<p<\infty\\
D(A_{\infty})&=\{u\in C_0(\R^N)\cap W^{4,q}_{\rm loc}(\R^N)\ \forall\ q<\infty:\ Vu,\   \Delta^2 u\in C_0(\R^N)\}.
\end{align*}
We endow these spaces with the norms
\begin{align*}
\|u\|_{D(A_1)}& =\|u\|_{1}+\|V u\|_{1}+ \|\Delta^2  u\|_{1} \\
\|u\|_{D(A_p)}& =\|u\|_{W^{4,p}}+\|V u\|_p,\quad 1<p<\infty\\
\|u\|_{D(A_{\infty})}& =\|u\|_{\infty}+\|V u\|_{\infty}+ \|\Delta^2 u\|_{\infty}.
\end{align*}

The following density result will allow us to work with smooth functions.

\begin{lemma}
Let  $1\leq p\leq \infty$ and assume that $V$ satisfies (\ref{gradient}). Then $C_c^\infty(\R^N)$ is dense in $D(A_1)$, $D(A_p)$  and $D(A_{\infty})$ with respect to the norms introduced above.
\end{lemma}
\begin{proof} 
Consider the case $1<p<\infty$. We first observe that if we take $u\in W^{4,p}(\R^N)$ with compact support, then we can approximate it in the norm of $D(A_p)$ with smooth functions by convolution with standard mollifiers. Therefore it remains to prove that a function in $D(A_p)$ can be approximated by a sequence of functions with compact support.
Let $\varphi\in C^\infty([0,\infty))$ such that $0\leq\varphi\leq1, \varphi(t)=1$ for $t\in[0,1]$ and $\varphi(t)=0$ for $t\geq2$ and then set 
\[\varphi_n(x)=\begin{cases} 1 & \textrm{in}\ B_n,\\\varphi\left(\frac{|x|}{n}\right) & \textrm{in}\ B_n^c.\end{cases}\]
 Observe that $\varphi_n\in C^\infty_c(\R^N)$ and 
\[\begin{cases} \varphi_n=0 \quad \textrm{in}\ B_{2n}^c,\\ \varphi_n=1\quad \textrm{in}\ B_n,\\0\leq\varphi_n\leq1,\\|D^h\varphi_n|\leq Cn^{-h}\quad \textrm{for}\ h=1,2,3,4.
\end{cases}\]
Let $u\in D(A_p)$ and set $u_n=u\varphi_n$.
It is easy to  show   that $u_n$ converges to $u$ with respect to the norm $||\cdot||_{D(A_p)}$. The same argument works for $p=1$ and $p=\infty$.
% One has $u_n\to u\in L^p(\R^N)$ by dominated convergence. For the first order term we observe that
% \begin{align*}
% |D u_n-D u|&= |D u\, \varphi_n+uD\varphi_n-D u|
% \\&\leq  (1-\varphi_n)a^\frac12|D u|+  a^\frac12 |u||D\varphi_n| 
% \end{align*}
% and the right-hand side of the inequality converges to 0 a.e. Moreover,  in view of \eqref{derphi},
% \[(1-\varphi_n)a^\frac12|D u|+ a^\frac12 |u||D\varphi_n|\leq  (1-\varphi_n)a^\frac12|D u|+C|u|\chi_{B_{2n}\setminus B_n}\in L^p(\R^N).\]
% Therefore $a^\frac12D u_n$ converges to $a^\frac12D u$  in $L^p(\R^N)$.
% As regards the other terms, for $h=2,3,4$, one has 
% \begin{align*}
% |a^{\frac{h}{2}} D^h u_n-a^{\frac{h}{2}} D^h u|&\leq  (1-\varphi_n)a^{\frac{h}{2}} |D^h u|+C \sum_{i=0}^{h-1}\sum_{\substack{|l|=i\\|m|=h-i}}  a^{\frac{h}{2}}|D^lu||D^m\varphi_n|\to0 \ \textrm{a.e.}
% \end{align*}
% and by \eqref{derphi}
% \begin{align*}(1-\varphi_n)a^{\frac{h}{2}} |D^h u|+C& \sum_{i=0}^{h-1}\sum_{|l|=i,|m|=h-i}  a^{\frac{h}{2}}|D^lu||D^m\varphi_n|\leq\\ &\qquad(1-\varphi_n)a^{\frac{h}{2}} |D^h u|+C\sum_{i=0}^{h-1}\sum_{\substack{|l|=i\\|m|=h-i}}  a^{\frac{h-m}{2}}|D^lu|\chi_{B_{2n}\setminus B_n}\in L^p(\R^N).\end{align*}
\end{proof}

\subsection{A priori estimates}

The following proposition is a  consequence of the previous lemmas and gives the announced a priori estimates.

\begin{proposition} \label{domain}
Let  $1\leq p\leq \infty$ and assume that $V$ satisfies (\ref{gradient}). Then, for $u\in D(A_p)$ the following inequality holds
$$\|\Delta^2  u\|_p+\|V u\|_p\leq C\|Au\|_p$$
where the  constant $C$ depends on $N,\ p$. 
\end{proposition}
\begin{proof}
Take $u\in C_c^\infty(\R^N)$. Let $0<\delta<1$ to be fixed later and let us set $\rho(x)=\frac{4}{\gamma}((1+\delta)^\frac{1}{4}-1)U(x)^{-\frac{1}{4}}$. Note that, as stated before,  by \eqref{gradient} it follows that $\rho(x)$ satisfies the assumptions of Proposition \ref{pr:covering} with Lipschitz constant smaller than $\frac{1}{2}$. We fix  $x_{0}\in \R^{N}$ and set $R=\rho (x_{0})=\frac{4}{\gamma}((1+\delta)^\frac{1}{4}-1)U(x_0)^{-\frac{1}{4}}$.
Observe that, if $x\in B_R(x_0)$ we have
$\left |U^{-\frac 14}(x)-U^{-\frac 14}(x_0)\right |
\leq ((1+\delta)^{\frac 14}-1)U^{-\frac 14}(x_0)
$ and then 
$(2-(1+\delta)^{\frac 14})U^{-\frac 14}(x_0)\leq U^{-\frac 14}(x)\leq (1+\delta)^{\frac 14}U^{-\frac 14}(x_0)$,
and since $(1-\delta)^\frac 14\leq 2-(1+\delta)^\frac 14$ 
\[(1-\delta)U(x)
\leq U(x_0)\leq (1+\delta)U(x).\]
Moreover
$$|U(x)-U(x_0)|\leq \delta U(x).$$
Taking into account the hypothesis on $V$ we obtain
\begin{equation}\label{eq:Vx-Vx_0}
\frac{1-\delta}{c_0^2}V(x)\leq V(x_0)\leq c_0^2(1+\delta)V(x)    
\end{equation}
and then
$V(x_0)-V(x)\leq V(x)(c_0^2(1+\delta)-1)$ and $V(x)-V(x_0)\leq V(x)\left (1-\frac{(1-\delta))}{c_0^2}\right)\leq V(x)(c_0^2(1+\delta)-1)$, so
\begin{equation*}\label{eq:abs-Vx-Vx_0}
|V(x)-V(x_0)|\leq V(x)(c_0^2(1+\delta)-1).  
\end{equation*}

\smallskip
Let now  $\vartheta\in C_c^{\infty }(\R^N)$ be such that $0\leq \vartheta \leq 1$,
$\vartheta(x)=1$ for $x\in B_{R/2}(x_{0}),\vartheta(x)=0$ for $x\in B^{c}_{R}(x_{0})$ and $|D^h\vartheta|\leq {C}{R^{-h}}$ for $h=1,2,3,4$. Let us start by the case $p<\infty$.
By the classical sectoriality estimate for the operator $-\Delta^2$ (see \cite{lun}) and the positivity of $V$, we have
\begin{align*}
\frac{(1-\delta)}{c_0^2}\|V(x) u\|_{p, R/2}&\leq \|V(x_0) u\|_{p, R/2}\leq \|V(x_0) \vartheta u\|_{p}
\leq C\|\Delta^2(\vartheta u)+V(x_0)\vartheta u\|_p\\ &\leq C(\|\Delta^2(\vartheta u)+V\vartheta u\|_p+ \|(V(x)-V(x_0))\vartheta u\|_p)\\ &\leq C\left(\|\vartheta Au\|_p+\sum_{k=1}^4\|D^k\vartheta D^{4-k}u\|_p+ \|(V(x)-V(x_0))\vartheta u\|_p\right)\\ &\leq C\left(\|Au\|_{p,R}+\sum_{k=1}^4\frac{1}{R^k}\| D^{4-k}u\|_{p,R}+ \|(V(x)-V(x_0)) u\|_{p,R}\right)\\ &\leq C\left(\|Au\|_{p,R}+\sum_{k=1}^4C(\delta)\gamma^kV(x_0)^\frac{k}{4}\| D^{4-k}u\|_{p,R}+ \|(V(x)-V(x_0)) u\|_{p,R}\right).
\end{align*}
By \eqref{eq:Vx-Vx_0} 
%and \eqref{eq:abs-Vx-Vx_0}
we obtain
\begin{align} \label{eq:cover-x000}
\frac{(1-\delta)}{c_0^2}\|V(x) u\|_{p, R/2}&\leq C\left(\|Au\|_{p,R}+\sum_{k=1}^4C(\delta)\gamma^k\| V(x)^\frac{k}{4}D^{4-k}u\|_{p,R}+ (c_0^2(1+\delta)-1)\|Vu\|_{p,R}\right).
\end{align}

Let $\{B_{\rho(x_n )/2}(x_n)\}$ be a countable covering of $\R^N$ as in Proposition \ref{pr:covering}
such that at most $\zeta $ among the double balls $\{B_{\rho(x_n )}(x_n)\}$ overlap.

We write \eqref{eq:cover-x000} with $x_0$ replaced by $x_n$ and sum over $n$.
Taking into account the above covering result, we get
\begin{align*} 
\frac{(1-\delta)}{c_0^2}\|V(x) u\|_{p}&\leq C\left(\|Au\|_{p}+\sum_{k=1}^4C(\delta)\gamma^k\| V(x)^\frac{k}{4}D^{4-k}u\|_{p}+ (c_0^2(1+\delta)-1)\|V u\|_{p}\right).
\end{align*}
It follows that
\begin{align*} 
\left( \frac{1-\delta}{c_0}-C(c_0^2(1+\delta)-1)\right)\|V(x) u\|_{p}&\leq C\left(\|Au\|_{p}+\sum_{k=1}^4C(\delta)\gamma^k\| V(x)^\frac{k}{4}D^{4-k}u\|_{p}\right).
\end{align*}
By choosing $\delta$ small enough and $c_0$ near to $1$ enough, we get
\begin{align*} 
\|V u\|_{p}&\leq C\left(\|Au\|_{p}+\sum_{k=1}^4\gamma^k\| V^\frac{k}{4}D^{4-k}u\|_{p}\right).
\end{align*}
Let us now estimate $\sum_{k=1}^4\gamma^k\| V^\frac{k}{4}D^{4-k}u\|_{p}$ by using Lemma  \ref{interpolation}.
For every $\varepsilon>0$ 
\begin{align*} 
\sum_{k=1}^4\gamma^k\| V^\frac{k}{4}D^{4-k}u\|_{p}&= \gamma\| V^\frac{1}{4}D^{3}u\|_{p}+\gamma^2\| V^\frac{1}{2}D^{2}u\|_{p}+\gamma^3\| V^\frac{3}{4}Du\|_{p}+\gamma^4\| Vu\|_{p}\\&\leq \gamma\left(\varepsilon \|\Delta^{2}u\|_{p}+\frac{C}{\varepsilon^{3}}\|Vu\|_{p})+\gamma^2(\varepsilon^2 \|\Delta^{2}u\|_{p}+\frac{C}{\varepsilon^2}\|Vu\|_{p}\right)\\&+\gamma^3\left(\varepsilon ^{3}\|\Delta^2u\|_{p}+\frac{C}{\varepsilon}\|Vu\|_{p}\right)+\gamma^4\| Vu\|_{p}
\end{align*}
By choosing $\varepsilon=\gamma^\frac{1}{6}$, we get 
\begin{align*} 
&\gamma\| V^\frac{1}{4}D^{3}u\|_{p}+\gamma^2\| V^\frac{1}{2}D^{2}u\|_{p}+\gamma^3\| V^\frac{3}{4}Du\|_{p}+\gamma^4\| Vu\|_{p}\\&\leq \gamma\left(\gamma^\frac{1}{6} \|\Delta^{2}u\|_{p}+\frac{C}{\gamma^\frac{1}{2}}\|Vu\|_{p})+\gamma^2(\gamma^\frac{1}{3} \|\Delta^{2}u\|_{p}+\frac{C}{\gamma^\frac{1}{3}}\|Vu\|_{p}\right)\\&+\gamma^3\left(\gamma^\frac{1}{2}\|\Delta^2u\|_{p}+\frac{C}{\gamma^\frac{1}{6}}\|Vu\|_{p})+\gamma^4\| Vu\|_{p}\leq C(\gamma)(\|Au\|_p+\|Vu\|_p\right)
\end{align*}
where $C(\gamma)=C\gamma^a$ for some $a>0$.
It follows that
\begin{align*} 
\|V u\|_{p}\leq C\|Au\|_{p}+C(\gamma)\|Vu\|_p
\end{align*} and the claim follows by choosing $\gamma$ small enough.
When $p=\infty$, we argue in similar way and we don't need to use the covering argument above.
\end{proof}

\begin{remark} \label{supp-compatto}
 Let $1<p<\infty$. The inequalities proved in Proposition \ref{interpolation} remain true for functions in $W^{4,p}_{\rm comp}(\R^N)$.
 \end{remark}

\subsection{Domain characterization for \texorpdfstring{$1<p<\infty$}{1<p<infty}}
 We are now able to characterize the domain of $A$ in $L^p$, $1<p<\infty$. Let us define, as usually, the maximal domain
 $$D(A_{p,\rm max})=\{u\in W^{4,p}_{\rm loc}(\R^N)\cap L^p(\R^N):\   Au\in L^p(\R^N)\}.$$
\begin{proposition} \label{domain2}
Let  $1<p<\infty$ and assume that $V$ satisfies (\ref{gradient}). Then 
$$D(A_{p})=D(A_{p,\rm max}).$$
\end{proposition}
 \begin{proof}
We have only to show the inclusion $D(A_{p,\rm max})\subseteq D(A_{p})$, being the other obvious. 
Take $x_0\in\R^N$, $R\in\R$ to be chosen later. 
Take  $\sigma\in (0,1)$ and set $\sigma':=\frac{\sigma+1}{2}$.
Consider a cutoff function $\vartheta\in \cc$ such that $0\leq\vartheta\leq 1$,
$\vartheta=1$ for $x\in B_{\sigma R}(x_{0})$, $\vartheta(x)=0$ for 
$x\in B^{c}_{\sigma'R}(x_{0})$ and $\|D^{h}\vartheta\|_{\infty}\leq {C}{R^{-h}(1-\sigma')^{-h}}$ for $h=0,1,2,3,4$.
Let $u\in D(A_{p,\rm max})$. In particular $u\in W_{\rm loc}^{4,p}(\R^{N})$. Since the classical interpolation inequalities for higher order derivatives and the inequalites proved in Proposition \ref{domain}, as already observed,  hold  for $\vartheta u\in W_{\rm comp}^{4,p}(\R^{N})$,   there exists a positive constant $C$ only depending on $N,p$ and $\gamma$ such that for $h=1,2,3$ the following inequalities hold
\[
\|D^h(\vartheta u)\|_{p}\leq \varepsilon ^{4-h}\|\Delta^{2}(\vartheta u)\|_{p}+\frac{C}{\varepsilon^h}\|(\vartheta u)\|_{p}\leq \varepsilon ^{4-h}\|A(\vartheta u)\|_{p}+\frac{C}{\varepsilon^h}\|(\vartheta u)\|_{p},
\]
and for $h=4$
\[
\|D^4(\vartheta u)\|_{p}\leq C\left( \|A(\vartheta u)\|_{p}+\|(\vartheta u)\|_{p}\right).
\]
It follows that for $h=1,2,3$ we have
\begin{align*}
 \|D^hu\|_{p,\sigma R} &\leq\varepsilon^{4-h} \|A(u\vartheta)\|_{p,\sigma'R}+\frac{C}{\varepsilon^h}\|u\|_{p,R}\\
& \leq \varepsilon^{4-h} \|Au\|_{p,R}+\varepsilon^{4-h} 
\big( \|D^{3}u D\vartheta\|_{p,\sigma'R}+\|D^{2}uD^{2}\vartheta\|_{p,\sigma'R}\\&\quad+\|DuD^{3}\vartheta\|_{p,\sigma'R}
	+\|uD^{4}\vartheta\|_{p,\sigma'R}  \big)+\frac{C}{\varepsilon^h}\|u\|_{p,R}\\
& \leq \varepsilon^{4-h} \|Au\|_{p,R}+\varepsilon^{4-h} 
\left( \frac{1}{R(1-\sigma')}\|D^{3}u \|_{p,\sigma'R}+\frac{1}{R^{2}(1-\sigma')^{2}}\|D^{2}u\|_{p,\sigma'R}\right.\\
&\quad \left.
	+\frac{1}{R^{3}(1-\sigma')^{3}}\|Du\|_{p,\sigma'R}
	+\frac{1}{R^{4}(1-\sigma')^{4}}\|u\|_{p,\sigma'R}  \right) +\frac{C}{\varepsilon^h}\|u\|_{p,R}.
\end{align*}
For $h=4$ we obtain
\begin{align}
 \|D^{4}u\|_{p,\sigma R }&\leq C\left(  \|Au\|_{p,R}+
 \frac{1}{R(1-\sigma')}\|D^{3}u \|_{p,\sigma'R}+\frac{1 }{R^2(1-\sigma')^{2}}\|D^{2}u\|_{p,\sigma'R}\right.\nonumber \\
&\quad \left.
	+\frac{1 }{R^3(1-\sigma')^{3}}\|Du\|_{p,\sigma'R}
	+\frac{1}{R^4(1-\sigma')^{4}}\|u\|_{p,\sigma'R} +\|u\|_{p,R} \right).\label{eq:quartordine}
\end{align}

Set $\varepsilon=\frac{1}{(2^h6C)^{\frac{1}{4-h}}}(1-\sigma')$, then one can find a suitable positive $C_1(R)$ such that 
for $h=1,2,3$
\begin{align*}
\|D^hu\|_{p,\sigma R}&\leq C_1(1-\sigma')^{4-h} \|Au\|_{p,R}+\frac{1}{2^h6}
\left( (1-\sigma')^{3-h}\|D^{3}u \|_{p,\sigma'R}+(1-\sigma')^{2-h}\|D^{2}u\|_{p,\sigma'R}\right.\\
&\qquad \left.
	+(1-\sigma')^{1-h}\|Du\|_{p,\sigma'R}
	+(1-\sigma')^{-h}\|u\|_{p,\sigma'R}  \right)+\frac{C_1}{(1-\sigma')^h}\|u\|_{p}.\end{align*}
Multiply by $(1-\sigma')^h$, taking into account that $1-\sigma'=\frac{1-\sigma}{2}$, it follows that
\begin{align*}\label{eq:interp.01}
(1-\sigma)^h\|D^hu\|_{p,\sigma R}&\leq C \|Au\|_{p,R}+\frac{1}{6}
\left( (1-\sigma')^{3}\|D^{3}u \|_{p,\sigma'R}+(1-\sigma')^{2}\|D^{2}u\|_{p,\sigma'R}\right.\nonumber \\
&\qquad \left.
	+(1-\sigma')\|Du\|_{p,\sigma'R}
	+\|u\|_{p,\sigma'R}  \right)+C\|u\|_{p}.
\end{align*}

Now, summing for $h=1,2,3$ one has
\begin{align*}
&
(1-\sigma)^{3}\|D^{3}u \|_{p,\sigma R}+(1-\sigma)^{2}\|D^{2}u\|_{p,\sigma R}+(1-\sigma)\|Du\|_{p,\sigma R}
 \leq C \|Au\|_{p,R}+C\|u\|_{p}\\
&\qquad +\frac{1}{2}\left( (1-\sigma')^{3}\|D^{3}u \|_{p,\sigma'R}+(1-\sigma')^{2}\|D^{2}u\|_{p,\sigma'R}
	+(1-\sigma')\|Du\|_{p,\sigma'R}  \right).
\end{align*}
Set $\Phi(u)=\sup_{\sigma\in (0,1)}\left\{(1-\sigma)^{3}\|D^{3}u \|_{p,\sigma R}+(1-\sigma)^{2}\|D^{2}u\|_{p,\sigma R}+(1-\sigma)\|Du\|_{p,\sigma R} \right\}$
then the above inequality reads as 
\[
\Phi(u)\leq \frac{1}{2}\Phi(u)+C\left( \|Au\|_{p,R}+\|u\|_{p}\right)
\]
and then
\begin{align*}
&\frac{1}{8}\left(
	\|D^{3}u \|_{p,\frac{1}{2}R}+\| D^{2}u\|_{p,\frac{1}{2} R}+\|Du\|_{p,\frac{1}{2} R}
	\right) \\
&\quad \leq \frac{1}{8}\|D^{3}u \|_{p,\frac{1}{2}R}+\frac 14 \|D^{2}u\|_{p,\frac{1}{2} R}+\frac 12\|Du\|_{p,\frac{1}{2} R} \\
&\quad \leq \Phi (u)\leq  C \|Au\|_{p,R}+C\|u\|_{p}.
\end{align*}
Therefore, for all $u \in W_{\rm loc}^{4,p}(\R^{N})$ we have
\begin{equation}\label{eq:cover-max}
\|D^{3}u \|_{p,\frac{1}{2}R}+\| D^{2}u\|_{p,\frac{1}{2} R}+\|Du\|_{p,\frac{1}{2} R}
\leq C\left( \|Au\|_{p,R}+\|u\|_{p,R} \right)
\end{equation}
and by \eqref{eq:quartordine}
\begin{equation}\label{eq:cover-max2}
\|D^{4}u\|_{p,\frac 12 R} \leq 
C\left(\|Au\|_{p,R}+  \|D^{3}u \|_{p,R}+\| D^{2}u\|_{p, R}+\|Du\|_{p,R}
+\|u\|_{p,R} \right).
\end{equation}

Now let $u\in W_{\rm loc}^{4,p}(\R^{N})\cap L^{p}(\R^{N})$ such that $Au\in L^{p}(\R^{N})$ and
let $\{B_{\frac{1}{2}R}(x_n)\}$ be a countable covering of $\R^N$ as in Proposition \ref{pr:covering}
such that at most $\zeta $ among the double balls $\{B_{R}(x_n)\}$ overlap.

We write \eqref{eq:cover-max} with $x_0$ replaced by $x_n$ and summing over $n$ we obtain
\[
\|D^{3}u \|_{p}+\| D^{2}u\|_{p}+\|Du\|_{p}
 \leq C\left( \|Au\|_{p}+\|u\|_{p} \right).
\]
Writing \eqref{eq:cover-max2} with $x_0$ replaced by $x_n$ and summing over $n$ we obtain
\begin{align*}
 \|D^{4}u\|_{p}&\leq 
C\left(\|Au\|_{p}+
\|D^{3}u \|_{p}+\| D^{2}u\|_{p}+\|Du\|_{p}+\|u\|_{p}
\right)	
\\
&  \leq C\left( \|Au\|_{p}+\|u\|_{p} \right)
\end{align*}
and therefore
\[\|u\|_{D(A_p)}\leq C(\|Au\|_p+\|u\|_p).\]

\end{proof}

\subsection{A priori estimates for an auxiliary operator}
To prove the sectoriality estimate for the resolvent as in \cite[Theorem 2.1]{agm62}, we need a priori estimates also for the elliptic operator $A_1=-\Delta^2-e^{i\theta}D_t^{(4)}-V(x)$, $\theta\in \left[-\frac{\pi}{2},\frac{\pi}{2}\right]$, in $L^p(\R^{N+1})$. The next results will be stated and proved for $1<p<\infty$. The proof for the extremal cases $1$ and $\infty$ will be different and separately studied.  
\begin{proposition} \label{a-prioriAux}
Let  $1<p<\infty$ and assume that $V$ satisfies (\ref{gradient}). Then there exists a positive constant $C=C(N,p,c)$ such that
\begin{align*}
 \|D^{(4)}_tu\|_{L^p(\R^{N+1})}+\|\Delta^2 u\|_{L^p(\R^{N+1})}+\|V u\|_{L^p(\R^{N+1})}\leq C\|A_1u\|_{L^p(\R^{N+1})}
\end{align*}
and 
\begin{align*}
\sum_{k=1}^4 \|D^{(k)}_tu\|_{L^p(\R^{N+1})}+\|\Delta^2 u\|_{L^p(\R^{N+1})}+\|V u\|_{L^p(\R^{N+1})}\leq C\left(\|A_1u\|_{L^p(\R^{N+1})}+\|u\|_{L^p(\R^{N+1})}\right)
\end{align*}
for every $u\in C_c^\infty(\R^{N+1})$.
\end{proposition}
\begin{proof} 
%Let $u\in C_c^\infty(\R^{N+1})$, $\theta\in \left[-\frac{\pi}{2},\frac{\pi}{2}\right]$. 
The claim follows as in Proposition \ref{domain} by localization and covering in $\R^N$ and by using the classical interpolation inequalites for the derivatives with respect to $t$.
\end{proof}

 \section{Generation results in \texorpdfstring{$L^p(\R^N)$}{Lp(RN)}, \texorpdfstring{$1<p<\infty$}{1<p<infty}}
% In this section we investigate the generation of an analytic semigroup in $L^p(\R^N)$ for $1<p<\infty$.
 Following Agmon, we first prove an a priori estimate for the resolvent.

\begin{theorem} \label{agmon-comp}
Let  $1<p<\infty$ and assume that $V$ satisfies (\ref{gradient}). Then, if $\lambda=re^{i\theta}$ for some $r>r_{p}$ and $\theta\in [-\frac{\pi}{2},\frac{\pi}{2}]$, the following inequality holds for every $u\in D(A_p)$
$$|\lambda|\|u\|_p\leq C\|\lambda u-Au\|_p$$
where $C=C(N,p,c)$.
 \end{theorem}
\begin{proof}
Let $u\in C_{c}^{\infty}(\R^{N})$ and set $v(t,x)=\psi (t)e^{ir t}u(x)$ where $r>0$ and $\psi\in C_c^\infty(\R)$ is such that $0\leq\psi\leq1,\psi(t)=1$ for $t\in \left[-\frac{1}{2},\frac{1}{2}\right],\ \psi(t)=0$ for $|t|\geq 1$ and $|D^h\psi|\leq {C} ,h=1,2,3,4$.
By Proposition \ref{a-prioriAux} 
\begin{align*}
\sum_{k=1}^4 \|D^{(k)}_tv\|_{L^p(\R^{N+1})}
	+\|\Delta^2 v\|_{L^p(\R^{N+1})}+\|Vv\|_{L^p(\R^{N+1})}\leq C(\|A_1v\|_{L^p(\R^{N+1})}+\|v\|_{L^p(\R^{N+1})}).
\end{align*}
Observe  that 
$$\|A_1 v\|_{L^p(\R^{N+1})}\leq\|\Delta ^2 u+e^{i\theta}r^4u+Vu\|_{L^p(\R^{N})}+C\sum_{k=0}^3r^{k}\|u\|_{L^p(\R^{N})}.$$
Therefore
\begin{align*}
\sum_{k=1}^4 &\|D^{(k)}_te^{irt}u\|_{L^p\left(\R^{N}\times\left[-\frac{1}{2},\frac{1}{2}\right]\right)}+\|e^{irt}\Delta^2 u\|_{L^p\left(\R^{N}\times\left[-\frac{1}{2},\frac{1}{2}\right]\right)}+\|e^{irt}V u\|_{L^p\left(\R^{N}\times\left[-\frac{1}{2},\frac{1}{2}\right]\right)}\\&
	\leq C\left(\|\Delta ^2u+r^4e^{i\theta}u+Vu\|_{L^p(\R^{N})}+\sum_{k=0}^3r^{k}\|u\|_{L^p(\R^{N})}\right).
\end{align*}
On the other hand
\begin{align*}
\sum_{k=1}^4 &\|D^{(k)}_te^{irt}u\|_{L^p\left(\R^{N}\times\left[-\frac{1}{2},\frac{1}{2}\right]\right)}= \sum_{k=1}^4 \|r^ku\|_{L^p\left(\R^{N}\right)}
\end{align*}
and then
\[
\left( r^{4}+(1-C)\sum_{k=0}^{3}r^{k} \right)\|u\|_{L^p\left(\R^{N}\right)}\leq C(\|\Delta ^2 u+r^4e^{i\theta}u+Vu\|_{L^p(\R^{N})}.
\]
Since 
$r^4-C\sum_{k=0}^3r^{k}\geq \frac{r^4}{2}$ for $r$ large enough, it follows that
$$r^4\|u\|_{L^p\left(\R^{N}\right)}\leq C\|\Delta ^2 u+r^4e^{i\theta}u+Vu\|_{L^p(\R^{N})}.$$ 
Setting $\lambda=r^{4}e^{i\theta}$ we have
$$|\lambda| \|u\|_{L^p\left(\R^{N}\right)}\leq C\|\Delta ^2 u+Vu+\lambda u\|_{L^p(\R^{N})}.$$
\end{proof}

% The previous estimate leads to the solvability of the elliptic problem. 
In order to get generation of an analytic semigroup from the sectoriality estimate above, we prove the existence of the solution of the elliptic problem.
\begin{theorem}\label{solv}
Let  $1<p<\infty$ and assume that $V$ satisfies (\ref{gradient}). Then there exists $r_p>0$ such that  the equation $\lambda u-Au=f$ is uniquely solvable in $D(A_p)$ for $f\in L^p(\R^N)$ and  $\Rp\lambda>r_p$.  
\end{theorem}
 \begin{proof}
 
 Let $\sigma>0$ and define the function
$$V_\sigma(x)=\frac{V(x)}{1+\sigma V(x)},\quad x\in\R^N.$$ Observe that, if $V(x)\geq c_0>0$,  $\frac{c_0}{1+\sigma c_0}\leq V_\sigma(x)\leq\sigma^{-1}$, $|D V_\sigma|\leq \gamma V_\sigma^\frac{5}{4}$ with  $\gamma$ as before and independent of $\sigma$. Consider the operators with bounded coefficients $A_\sigma=-\Delta^2-V_\sigma$. 
By \cite[Theorem 3.2.2]{lun}, there exist $\omega_{\sigma,p}\in\R$, $M_{\sigma,p}>0$ such that $\{\lambda\in\C:\ \Rp \lambda\geq \omega_{\sigma,p}\}\subset\rho(A_{\sigma})$ and 
\begin{equation*} %\label{stima-dipendente}
|\lambda|\|u\|_p\leq M_{\sigma,p}\|\lambda u-A_\sigma u\|_p
\end{equation*} for every $u\in W^{4,p}(\R^N)$  and $\lambda\in\C$ with $\Rp \lambda\geq \omega_{\sigma,p}$.
Morever, by applying Theorem  \ref{agmon-comp}  to $A_\sigma$, since the condition in (\ref{gradient}) is satisfied with a constant independent of $\sigma$, we have that there exists $r_p>0$, $C_p>0$  (independent of $\sigma$) such that  for every  $\lambda\in\C$ with $\Rp\lambda>r_p$ and for every $u\in C_c^\infty(\R^N)$, and then in $W^{4,p}(\R^N)$ by density, 
\begin{equation} \label{stima-ind}
|\lambda|\|u\|_p\leq C_p\|\lambda u-A_\sigma u\|_p
\end{equation}
for some constant $C_p$ independent of $\sigma$.  Let $\overline{\lambda}\in \rho (A_{\sigma })$ and $\Rp\overline{\lambda}>r_p$, then the inequality (\ref{stima-ind}) gives that
\begin{equation} \label{resolvent}
\|R(\overline{\lambda},A_\sigma)\|_p\leq \frac{C_p}{|\overline{\lambda}|}\leq \frac{C_p}{r_p}.
\end{equation}
By classical results, if $|\lambda-\overline{\lambda}|<\frac{1}{\|R(\overline{\lambda},A_\sigma)\|_p}$, then  $\lambda\in\rho (A_{\sigma })$. By (\ref{resolvent}), if
$|\lambda-\overline{\lambda}|<\frac{r_p}{C_p}$, then  $\lambda\in\rho (A_{\sigma })$.  
By covering with balls of radius $\frac{r_p}{C_p}$, it follows that, if $\Rp\lambda>r_p$, then $\Rp\lambda\in\rho(A_{\sigma })$.

Let $\Rp\lambda>r_p$, $f\in L^p(\R^N)$. Denote by $u_\sigma\in W^{4,p}(\R^N)$ the unique solution of $\lambda u-A_\sigma u=f$. Then $\{\lambda u_\sigma\}$ is bounded in $L^p(\R^N)$ and from $A_\sigma u_\sigma=\lambda u_\sigma -f$ we have that also $\{A_\sigma u_\sigma\}$ is bounded in $L^p(\R^N)$. Moreover, as in Proposition \ref{domain}, we have that
\begin{equation*}\label{a-priori-sigma}
\| \Delta^2 u_\sigma\|_p+\|V_\sigma u_\sigma\|_p\leq C\|\lambda u_\sigma-A_\sigma u_\sigma\|_p=C\|f\|_p\end{equation*}  for some constant $C$ independent of $\sigma$. By the classical interpolation inequalities, 
$$\sum_{h=0}^4|\lambda|^{1-\frac{h}{4}}\|D^hu_\sigma\|_p\leq C\left(\|\Delta^2 u_\sigma\|_p+\|\lambda u_\sigma\|_p\right)\leq C\|f\|_p.$$
Therefore, there exists a subsequence $u_{\sigma_n}$ and a function $u$ such that $u_{\sigma_n}$ converges weakly in $W^{4,p}(\R^N)$ and strongly in $W^{3,p}(\R^N)$ to $u$. Moreover  $u_{\sigma_n}$, $D^{(k)}u_{\sigma_n}$ converge pointwise almost everywhere to $u$, $D^{(k)}u$ for every $1\leq k\leq 3$ and hence $\|Vu\|_p\leq C\|f\|_p$ and $u\in D_p(A)$. Passing to the limit in $L^p_{\textrm{loc}}$ in the equation $\lambda u_{\sigma_n}+\Delta^2u_{\sigma_n}+Vu_{\sigma_n}=f$, we get that $u$ solves
 $$\lambda u+\Delta^2u+Vu=f.$$
This proves that $u\in D(A_p)$ solves $\lambda u-Au=f$. 
\end{proof}

We proved that the operator $A$ is sectorial and that a suitable half plain is in the resolvent set. By standard classical generation results, cf. \cite[Ch. II, Theorem 4.6]{eng-nag}, we have the following generation theorem.
\begin{theorem}
Let  $1<p<\infty$ and assume that $V$ satisfies (\ref{gradient}). Then, $(A,D_p(A))$ generates an analytic semigroup in $L^p(\R^N)$.
\end{theorem}

From the sectoriality and the interpolation estimates, we deduce the following result.
\begin{coro}\label{agmon-comp-2}
Let  $1<p<\infty$ and assume that $V$ satisfies (\ref{gradient}). Then there exist $r_p>0$ and $C=C(N,p)$ such that  for every  $\Rp\lambda>r_p$ and $u\in D(A_p)$
\begin{align*}
&|\lambda|\|u\|_p+|\lambda|^\frac{3}{4}\|Du\|_{p}+|\lambda|^\frac{1}{2}\|D^{2}u\|_{p}+|\lambda|^\frac{1}{4}\|D^{3}u\|_{p}+ \|D^4u\|_p+\\&\quad\quad\quad\|V^\frac{3}{4}Du\|+\|V^\frac{1}{2}D^2u\|_p+\|V^\frac{1}{4}D^3u\|_p+\|Vu\|_p\leq C \|\lambda u-Au\|_p.
\end{align*}
\end{coro}
\begin{proof} The proof follows by Lemma \ref{interpolation}, Theorem \ref{agmon-comp},  Proposition \ref{domain} 
 and the classical interpolation estimates.\end{proof}

% \begin{remark}
% We observe   that, by perturbation techniques, one can recover the same generation result for a more general operator containing lower order terms with unbounded coefficients satisfying growth bounds. Indeed, if the $h$ order coefficients, say $b_h(x)$,  satisfy  for every $h=1,2,3$
% \[|b_h(x)|\leq k a(x)^{\frac{h}{2}}\qquad\forall\,x\in\R^N\]
% for a positive constant $k$, then the resulting operator is, by Lemma \ref{interpolation}, $A$-bounded and one can apply \cite[Theorem 2.10]{eng-nag}.
% \end{remark}

\section{Generation results in continuous functions spaces}

 In this section  we use Stewart-Masuda localization technique to  prove generation results in the space $C_0(\R^N)$ of continuous functions going to zero as $|x|$ goes to infinity. 
 %To this purpose we assume   that there exist   some positive constants $\nu$ and $c$ such that
% \begin{equation} \label{gradient-infty}\tag{$H_{\infty}$}
% a  \in C^1(\R^N),\quad a(x)\geq \nu>0\ {\rm for\ all}\ x\in\R^N,\quad |D a(x)|\leq c a(x)^\frac{1}{2}.
% \end{equation}
We endow the operator $A$ in $C_0(\R^N)$ with the following domain
$$D(A_{\infty})=\{u\in C_0(\R^N)\cap W^{4,q}_{\rm loc}(\R^N)\ \forall\ q<\infty:\ Vu,\   \Delta^2 u\in C_0(\R^N)\} $$ and

$$
\|u\|_{D(A_{\infty})}=\|u\|_{\infty}+\|V u\|_{\infty}+ \|\Delta^2  u\|_{\infty}.$$

We recall the  Morrey embedding, cf. \cite[Lemma 2.1]{bro-opi}. 
\begin{lemma}[Morrey] If $p>N$, then $W^{1,p}(\R^{N})$ is continuously embedded in $C_{b}(\R^{N})$. 
Moreover
\begin{equation}\label{morrey}
\|g\|_{L^{\infty} (B_{R}(x_{0})) }\leq CR^{-\frac Np}\left( \|g\|_{L^{p} (B_{R}(x_{0}))}+R\|Dg\|_{L^{p} (B_{R}(x_{0}))} \right)
\end{equation}
for every $g\in W^{1,p}_{\rm loc}(\R^{N})$, $x_{0}\in \R^N$ and $R>0$.	
\end{lemma}
	
We deduce the following estimates.	
\begin{proposition}\label{prop:infty}
Let $p>N$ and assume that $V$ satisfies (\ref{gradient}). Then there exists $ r_p>0$ such that for every  $\Rp\lambda>r_p$ and
$u\in C^{3}_{0}(\R^{N})\cap W^{4,p}_{\rm loc}(\R^{N})$ with $Au \in C_{0}(\R^{N})$, it holds
\begin{align*} 
 \sum_{h=0}^{3}|\lambda|^{1-\frac h4} \|D^{h}u\|_{\infty} 
	+\sup_{x\in \R^{N}}  |\lambda|^{\frac {N}{4p}} \|D^{4}u\|_{ L^{p}(B_{R }(x))  }\leq C\|\lambda u-Au\|_{\infty}
\end{align*}
where $R=|\lambda |^{-\frac 14}$.
\end{proposition}
\begin{proof}
Let $u$ be as in the hypotheses of the proposition,   $\Rp\lambda>r_p$ as in Corollary \ref{agmon-comp-2} and let $\delta>0$ be a constant which   will be chosen later.
Set $R={|\lambda |^{-\frac 14}}$.
Let $x_0\in\R^N$, $\vartheta \in C_{c}^{\infty}(\R^{N})$ be such that $0\leq \vartheta \leq 1$, $\vartheta =1$ in $B_{R}(x_{0})$, $\vartheta=0$ in
$B^{c}_{R(1+\delta)} (x_{0})$ and $|D^{h}\vartheta|\leq {C}{R^{-h}\delta ^{-h}} $ for $h=1,2,3,4$  for some positive constant $C$ which does not depend on $x_{0}$ and $\delta.$ 
% Since $|D V^{-\frac{1}{4}}|\leq \frac{\gamma}{4}$ 
% we have that
% \[
% \left(1-\frac{\gamma(1+\delta)}{4|\lambda|^\frac{1}{4}}\right)V(x)^\frac{1}{4}\leq V(x_0)^{\frac 14}\leq \left(1+\frac{\gamma(1+\delta)}{4|\lambda|^\frac{1}{4}}\right)V(x)^\frac{1}{4}
% \text{ for every }x\in B_{R(1+\delta)}(x_{0})
% \]
%For $h=1,2,3,4$
% \[
% a(x)^{\frac h2}\leq \kappa^{h}a(x_{0})^{\frac h2} \text{ for every }x\in b_{r(1+\delta)}(x_{0});
% \] 
% moreover, choosing $|\lambda|\geq c^{4}$ we have
% \[
% a(x_{0})^{\frac 12}\leq \left(1-\frac{c}{2}\frac{1}{|\lambda|^{\frac 14}}\right)^{-1}a(x)^{\frac{1}{2}}\leq 2 a(x)^{\frac{1}{2}}
% \text{ for every }x\in b_{r}(x_{0})
% \]
% and then for $h=1,2,3,4$
% \[
% a(x_{0})^{\frac h2}\leq 2^{h} a(x)^{\frac{h}{2}}\text{ for every }x\in B_{R}(x_{0}).
% \]
Now for $h=0,1,2,3$ from \eqref{morrey} we have
\begin{align*}
 |\lambda| ^{1-\frac{h}{4}}|D^{h}u(x_{0})|
	&\leq CR^{-\frac Np}\left( |\lambda| ^{1-\frac{h}{4}} \|D^{h}u\|_{p,R} 
	+ |\lambda| ^{1-\frac{h}{4}} R\|D^{h+1}u\|_{p,R} \right) \\
&  \leq CR^{-\frac Np}\left( |\lambda| ^{1-\frac{h}{4}} \|D^{h}u\|_{p,R} 
	+ |\lambda| ^{1-\frac{h+1}{4}} \|D^{h+1}u\|_{p,R} \right). 	\end{align*}
Summing over $h=0,1,2,3$ we have
\begin{equation*}
\sum_{h=0}^{3}|\lambda| ^{1-\frac{h}{4}}|D^{h}u(x_{0})|
	\leq CR^{-\frac Np}\sum_{h=0}^{4}  |\lambda| ^{1-\frac{h}{4}} \| D^{h}u\|_{p,R} 
\end{equation*}
and also
\begin{equation}\label{eq:sum-lifnty1}
\sum_{h=0}^{3}|\lambda| ^{1-\frac{h}{4}}|D^{h}u(x_{0})|
		+|\lambda |^{\frac {N}{4p}}\|D^{4}u\|_{p,R}
	\leq CR^{-\frac Np}\sum_{h=0}^{4}  |\lambda| ^{1-\frac{h}{4}} \|D^{h}u\|_{p,R}.
\end{equation}
Now, setting $v=\vartheta u$, from Corollary \ref{agmon-comp-2} we have
\begin{align*}
 \sum_{h=0}^{4}  |\lambda| ^{1-\frac{h}{4}} \| D^{h}u\|_{p,R}&\leq 
	\sum_{h=0}^{4}  |\lambda| ^{1-\frac{h}{4}} \| D^{h}v\|_{p,R(1+\delta)}\\
&  \leq C\|\lambda v-Av\|_{p}\\
&  \leq C\left( \| \lambda u-Au\|_{p,R(1+\delta)} 
	+\sum _{h=0}^{3}\|D^{h}uD^{4-h}\vartheta\|_{p,R(1+\delta)} \right)\\
& \leq C\left( \| \lambda u-Au\|_{p,R(1+\delta)} 
	+\sum _{h=0}^{3} \frac{1}{R^{4-h}\delta ^{4-h}}    \|D^{h}u\|_{p,R(1+\delta)} \right)\\
&  \leq C\left( \| \lambda u-Au\|_{p,R(1+\delta)} 
	+\sum _{h=0}^{3} \frac{1}{\delta ^{4-h}}|\lambda |^{1-\frac h4}     \|D^{h}u\|_{p,R(1+\delta)} \right)\\
&  \leq C(1+\delta)^{\frac Np} R^{\frac Np}\left( \| \lambda u-Au\|_{\infty} 
	+\sum _{h=0}^{3} \frac{1}{\delta ^{4-h}}|\lambda |^{1-\frac h4}     \|D^{h}u\|_{\infty} \right).
\end{align*}
Then, by \eqref{eq:sum-lifnty1}, we have
\begin{align*}
\sum_{h=0}^{3}|\lambda| ^{1-\frac{h}{4}}&|D^{h}u(x_{0})|
		+|\lambda |^{\frac {N}{4p}}\|D^{4}u\|_{p,R} \\
&\quad \leq C(1+\delta)^{\frac Np} \left( \| \lambda u-Au\|_{\infty} 
	+\sum _{h=0}^{3} \frac{1}{\delta ^{4-h}}|\lambda |^{1-\frac h4}     \|D^{h}u\|_{\infty} \right).
\end{align*}
Taking the supremum over $x_{0}\in \R^{N}$ we have
\begin{align*}
\sum_{h=0}^{3}|\lambda| ^{1-\frac{h}{4}}&\|D^{h}u\|_{\infty}
		+\sup_{x_{0}\in \R^{N}}|\lambda |^{\frac {N}{4p}}\|D^{4}u\|_{p,R} \\
&\quad \leq C(1+\delta)^{\frac Np} \left( \| \lambda u-Au\|_{\infty} 
	+\sum _{h=0}^{3} \frac{1}{\delta ^{4-h}}|\lambda |^{1-\frac {h}{4}}     \|D^{h}u\|_{\infty} \right).
\end{align*}
After choosing $\delta_{0}>1$ such that $C\frac{(1+\delta_{0})^{\frac Np}}{\delta_{0}}\leq \frac 14$, the claim follows.

\end{proof}
We finally deduce the  generation result in $C_0(\R^N)$.
\begin{theorem}\label{generatcb}
Assume that $V$ satisfies (\ref{gradient}). Then $(A,D(A_{0}))$ generates an analytic semigroup in 
$C_{0}(\R^N)$.
\end{theorem}
\begin{proof}
% We prove that $(A_{0},D(A_{0}))$ is sectorial.
 Since we have the sectoriality estimate, we prove that the resolvent set is non empty.
Let $f\in C_{0}(\R^{N})$ and $f_{n}=\varphi_{n}f$ where $\varphi_{n}\in C_{c}^{\infty}(\R^{N})$ is such that 
$0\leq \varphi_{n}\leq 1$, $\varphi_{n}=1$, in $B_{n}$ and $\varphi_{n}=0$ in $B^{c}_{2n}$, so we have that $f_n$ converges uniformly to $f$ in $\R^N.$ 
Let $p>N$, since $f_{n}\in L^{p}(\R^{N})$ then, by Theorem \ref{solv}, for $Re\lambda > r_{p} $
there exists a unique $u_{n}\in D(A_{p})$ such that
\begin{equation}\label{eq:solvn}
\lambda u_{n}-Au_{n}=f_{n}.
\end{equation}
 We have, in particular, $u_{n}\in W^{4,p}(\R^{N})$ 
 and hence $D^hu_n\in C_0(\R^N)$ for $h=0, 1,2,3$ by Sobolev embedding theorem. 

By Proposition \ref{prop:infty} applied to $u_n-u_m$ we obtain
\begin{align*}
 \sum_{h=0}^{3}&|\lambda|^{1-\frac h4} \|D^{h}(u_{n}-u_m)\|_{\infty}
 +\sup _{x\in \R^N}\|D^4(u_n-u_m)\|_{L^p(B_{|\lambda|^{-1/4}}(x))}
 \leq C\|f_{n}-f_m\|_{\infty}.\label{eq:stima-infty}
\end{align*}
Then $\{u_{n}\}$ is a Cauchy sequence in $C_0^{3}(\R^{N})\cap W^{4,p}_{\rm loc}(\R^{N})$, therefore it converges to a function $u\in C^{3}_{0}(\R^{N})\cap W^{4,p}_{\rm loc}(\R^{N})$, $D^hu_n$ converges uniformly on $\R^N$ to $D^hu$ for $h=0,1,2,3$ and $D^4u_n$ converges locally in $L^p(\R^N)$ to $D^4u$. 
Taking the limit in \eqref{eq:solvn} we have that $u$ solves weakly the equation $\lambda u-Au=f$.
We now show that $Vu\in C_0(\R^N)$. Take $x_0\in\R^N$, $R>0$, $\eta\in C^\infty(\R^N)$ such that $0\leq \eta\leq 1$, $\eta=1$ in $B_R(x_0)$, $\eta=0$ outside $B_{2R}(x_0)$, $|D^k\eta|\leq CR^{-k}$ for $k=1,2,3,4$. Then $\eta u\in D(A_{0})$ and, by Proposition \ref{domain},
$$|V(x_0)u(x_0)|\leq \|V\eta u\|_\infty\leq C\left(\|\eta f\|_{\infty, 2R}+\sum_{k=1}^3\frac{1}{R^k}\|D^{4-k}u\|_{\infty,2R}
\right).
$$
Since $f\in C_0(\R^N),\ D^{h}u\in C_{0}(\R^N)$ for $h=0,1,2,3$, we deduce $Vu\in C_0(\R^N)$ and, by difference, also $\Delta^2u\in C_0(\R^N)$.
Finally, since $f\in L^{q}_{\rm loc}(\R^{N})$ for every $q>1$, by local elliptic regularity, $u\in W^{4,q}_{\rm loc}(\R^{N})$for every $q>1$ and then $u\in D(A_{0}).$
The clam follows by classical generation results.
\end{proof}

\section{Generation results in  \texorpdfstring{$L^1(\R^N)$}{L1(RN)}}

We consider the operator $A=-\Delta^{2}-V$ endowed with the domain
$$
D(A_1)
=\{u\in  L^1(\R^N):
\ \ Vu\in L^1(\R^N),\ \Delta ^2 u\in L^1(\R^N)\}.
$$
By \cite[Theorem 5.8 (i)]{tan}, $D(A_1)$ continuously embeds into $W^{3,p}(\R^N)$ for every $p\in \left[1,\frac{N}{N-1}\right[$ and, for every $u\in D(A_1)$,
\begin{equation*}\label{embT}
\|u\|_{W^{3,p}}\leq C_p(\|u\|_1+\|\Delta^2 u\|_1).
\end{equation*}
By Proposition \ref{domain} we have the following a priori estimates.
\begin{proposition} \label{domain-uno}
Let  $p=1$ and assume that $V$ satisfies (\ref{gradient}). Then, for $u\in D(A_1)$ it holds
$$\|\Delta^2  u\|_1+\|V u\|_1\leq C\|Au\|_1$$
where the  constant $C$ depends on $N$. 
\end{proposition}
To investigate the generation of an analytic semigroup in $L^1(\R^N)$, we start by proving an a priori sectoriality estimate for the resolvent in a suitable sector of the complex plane. Then we prove the surjectivity.

\begin{theorem} \label{sectorial}
Assume that $V$ satisfies (\ref{gradient}). Then, there exists $\Sigma\subseteq\C$ such that, if $\lambda\in\Sigma$, the following inequality holds for every $u\in D(A_1)$
$$|\lambda|\|u\|_1\leq C\|\lambda u-Au\|_1$$
where $C=C(N)$.
 \end{theorem}
\begin{proof}
We use a duality argument to get the sectoriality estimate in $L^1$.
Observe that the operator coincides with its formal adjoint.
By the conditions \eqref{gradient} on $V$, the operator $A$ generates an analytic semigroup in $C_0(\R^N)$  (Theorem \ref{generatcb}). Therefore there exists a sector $\Sigma$ contained in the resolvent set of $A$ and
$$|\lambda|\|u\|_\infty\leq C\|\lambda u-Au\|_\infty$$ for every $\lambda\in\Sigma$, $u\in D(A_{\infty}) $ and for some constant $C$ independent of $\lambda$ and $u$. In view of the density of $C_c^\infty(\R^N)$ in $D(A_1)$, we prove the sectoriality estimate for functions in $C_c^\infty(\R^N)$.
Let $\lambda\in\Sigma$, $u\in C_c^\infty(\R^N)$. Then, since $(\lambda I-A)(D(A_{\infty}))=C_0(\R^N)$, we have
\begin{align*}
\|u\|_1&= \sup\left\{\int_{\R^N}u(\lambda I-A)v\, dx:\ v\in D(A_{\infty}),\ \|\lambda v-A v\|_\infty\leq 1\right\}\\
       &\leq \sup\left\{\int_{\R^N}u(\lambda I-A)v\, dx:\ v\in D(A_{\infty}),\ \|v\|_\infty\leq \frac{C}{|\lambda|}\right\}.
\end{align*}
Integrating by parts, we get
\begin{align}
\left|\int_{\R^N}u(\lambda I-A)v\right|= \left|\int_{\R^N}v(\lambda I-A)u\right|\leq \frac{C}{|\lambda|}\|\lambda u-Au\|_1.
\end{align}
\end{proof}

We finally prove, through an approximation procedure,  the solvability of the elliptic problem. 

\begin{theorem}\label{solv-1}
Assume that $V$ satisfies (\ref{gradient}). Then there exists $\Sigma\subseteq\C$ such that  the equation $\lambda u-Au=f$ is uniquely solvable in $D(A_1)$ for $f\in L^1(\R^N)$ and  $\lambda\in\Sigma$  and the resolvent estimate 
$$|\lambda|\|u\|_1\leq C\|\lambda u-Au\|_1$$ holds.
\end{theorem}
\begin{proof} 
Let $\sigma>0$ and define the function
$$V_\sigma(x)=\frac{V(x)}{1+\sigma V(x)},\quad x\in\R^N.$$ Observe that $V_\sigma(x)$ satisfies (\ref{gradient}) with  $\gamma$ independent of $\sigma$. Consider the operators with bounded coefficients $A_\sigma=-\Delta^2-V_\sigma(x)$. 
By \cite[Section 5.4]{tan}, there exist $\omega_{\sigma}\in\R$, $M_{\sigma}>0$ such that $\{\lambda\in\C:\ \Rp \lambda\geq \omega_{\sigma}\}\subset\rho(A_{\sigma})$ and 
\begin{equation*} %\label{stima-dipendente}
|\lambda|\|u\|_1\leq M_{\sigma}\|\lambda u-A_\sigma u\|_1
\end{equation*} for every $u\in D(A_{\sigma,1,\rm max})
=\{u\in L^1(\R^N)\,| Au\in L^1(\R^N) \}$ 
and $\lambda\in\C$ with $\Rp \lambda\geq \omega_{\sigma}$.
Moreover  if $u\in D(A_{\sigma,1,\rm max})$, since $V_\sigma$ is bounded we have that $V_\sigma u\in L^1(\R^N)$ and hence $\Delta^2u\in L^1(\R^N) $ so $u\in D(A_{\sigma,1})$  and then we can apply Theorem  \ref{sectorial}  to $A_\sigma$. Since the condition in (\ref{gradient}) is satisfied with a constant independent of $\sigma$, we have that there exists $r_1>0$  (independent of $\sigma$) such that  for every  $\lambda\in\C$ with $\Rp\lambda>r_1$ and for every $u\in D(A_{\sigma,1,\rm max})$
\begin{equation} \label{stima-ind1}
|\lambda|\|u\|_1\leq C_1\|\lambda u-A_\sigma u\|_1
\end{equation}
for some constant $C_1$ independent of $\sigma$.  Let 
$\overline{\lambda}\in \rho (A_{\sigma })$ and $\Rp\overline{\lambda}>r_1$, then the inequality (\ref{stima-ind1}) gives that
\begin{equation} \label{resolvent1}
\|R(\overline{\lambda},A_\sigma)\|_1\leq \frac{C_1}{|\overline{\lambda}|}\leq \frac{C_1}{r_1}.
\end{equation}
By classical results, if $|\lambda-\overline{\lambda}|<\frac{1}{\|R(\overline{\lambda},A_\sigma)\|_1}$, then  $\lambda\in\rho (A_{\sigma })$. By (\ref{resolvent1}), if
$|\lambda-\overline{\lambda}|<\frac{r_1}{C_1}$, then  $\lambda\in\rho (A_{\sigma })$.  
By covering with balls of radius $\frac{r_1}{C_1}$, it follows that, if $\Rp\lambda>r_1$, then $\lambda\in\rho(A_{\sigma })$.
In particular, $\lambda=\overline{\lambda}-\frac{\overline{\lambda}}{2C_1}$ satisfies $|\lambda-\overline{\lambda}|<\frac{|\overline{\lambda}|}{C_1}$ and then  belongs to $\rho (A_\sigma)$.
Let $\Rp\lambda>r_1$, $f\in L^1(\R^N)$. Denote by $u_\sigma\in D(A_{\sigma,1,\rm max})$ the unique solution of $\lambda u_\sigma-A_\sigma u_\sigma=f$.
Then, by the a priori estimates, 
\begin{align}\label{est-sigma}
|\lambda|\|u_\sigma\|_1\leq C\|f\|_1,\  \|V_\sigma u_\sigma\|_1\leq C\|f\|_1,\  
\end{align}
with a constant $C$ independent of $\sigma$ since $V_\sigma$ satisfies (\ref{gradient}) with a constant uniform with respect to $\sigma$. It follows that  $\|\Delta^2 u_\sigma\|_1\leq C\|f\|_1$. 
By \cite[Theorem 5.8 (i)]{tan}, $D(A_{\sigma,1,\rm max})$ continuously
embeds into $W^{3,p}(\R^N)$ for every $p\in \left[1,\frac{N}{N-1}\right[$. Therefore, 
\begin{equation}\label{emb}
\|u_\sigma\|_{W^{3,p}}\leq C_p(\|u_\sigma\|_1+\|\Delta^2 u_\sigma\|_1)\leq  C_p(r_1)\|f\|_1.
\end{equation}
From (\ref{emb}) we deduce that, up to considering a subsequence, we can assume that $u_\sigma$ converges weakly in $W^{3,p}(\R^N)$ to some function $u\in W^{3,p}(\R^N)$ as $\sigma$ goes to $0$ and up to a subsequences $u_\sigma$ converges to $u$ in $L^1_{\rm loc}(\R^N)$ and then $u_\sigma$ converges to $u$ almost everywhere. From (\ref{est-sigma}), we deduce that $u\in L^1(\R^N)$ and
 \begin{equation*} 
|\lambda|\|u\|_1\leq C_1\|f\|_1.
\end{equation*}
Moreover $\Delta^2u_\sigma\to\Delta^2 u$ in $L^1_{\rm loc}(\R^N)$ 
and therefore $\lambda u+\Delta^2u+Vu=f$ and $u$ belongs to $D(A_1)$.
\end{proof}

 By the previous result and classical generation theorems we get the following. 

 \begin{theorem}
Assume that $V$ satisfies (\ref{gradient}). Then, $(A,D_1(A))$ generates an analytic semigroup in $L^1(\R^N)$.
\end{theorem}

%\bibliographystyle{amsplain}
%\bibstyle{plain}

%\bibliography{bibfile}
\bibliographystyle{elsarticle-num} 
\bibliography{bibfile}

\begin{thebibliography}{10}
\expandafter\ifx\csname url\endcsname\relax
  \def\url#1{\texttt{#1}}\fi
\expandafter\ifx\csname urlprefix\endcsname\relax\def\urlprefix{URL }\fi
\expandafter\ifx\csname href\endcsname\relax
  \def\href#1#2{#2} \def\path#1{#1}\fi

\bibitem{agrt}
D.~Addona, F.~Gregorio, A.~Rhandi, C.~Tacelli, Bi-{K}olmogorov type operators and weighted {R}ellich's inequalities, Nonlinear Differ. Equ. Appl. 29~(2) (2022) Paper No. 13, 37 pp.
\newblock \href {https://doi.org/https://doi.org/10.1007/s00030-021-00747-y} {\path{doi:https://doi.org/10.1007/s00030-021-00747-y}}.

\bibitem{gre-tac}
F.~Gregorio, C.~Tacelli, Fourth-order {S}chrödinger type operator with unbounded coefficients in ${L}^2(\mathbb{R}^{N})$, Math Meth Appl Sci. 45 (2022) 11246–11261.

\bibitem{gre-spi-tac1}
F.~Gregorio, C.~Spina, C.~Tacelli, Fourth-order operators with unbounded coefficients, Communications on Pure \& Applied Analysis 23 (2024) 463--479.

\bibitem{gre-spi-tac2}
F.~Gregorio, C.~Spina, C.~Tacelli, Fourth-order operators with unbounded coefficients in ${L}^1$ spaces, Mathematical models for interacting dynamics on networks, 375–388. Trends Math. Birkhäuser/Springer, Cham, (2026).
\newblock \href {https://doi.org/https://doi.org/10.1007/978-3-032-02326-1$\_$14} {\path{doi:https://doi.org/10.1007/978-3-032-02326-1$\_$14}}.

\bibitem{dav-hin98b}
E.~B. Davies, A.~M. Hinz, {K}ato class potentials for higher order elliptic operators, J. London Math. Soc. (2) 58~(3) (1998) 669–678.

\bibitem{gre-spi-tac3}
F.~Gregorio, C.~Spina, C.~Tacelli, Higher order {S}chrödinger operators, preprint, arXiv:2604.25547 (2026).

\bibitem{kato86}
T.~Kato, ${L}^p$-theory of {S}chr{\"o}dinger operators with a singular potential, North-Holland Math. Stud., 122 Notas Mat., North-Holland Publishing Co., Amsterdam 108 (1986) 63–78.

\bibitem{oka84}
N.~Okazawa, An ${L}^p$ theory for {S}chr{\"o}dinger operators with nonnegative potentials, J. Math. Soc. Japan. 36~(4) (1984) 675--688.
\newblock \href {https://doi.org/doi:10.2969/jmsj/03640675} {\path{doi:doi:10.2969/jmsj/03640675}}.

\bibitem{met-pru-rha-sch}
G.~Metafune, J.~Prüss, R.~Schnaubelt, A.~Rhandi, ${L}^p$-regularity for elliptic operators with unbounded coefficients, Adv. Differential Equations 10~(10) (2005) 1131–1164.

\bibitem{aus-ben}
P.~Auscher, B.~{Ben Ali}, Maximal inequalities and {R}iesz transform estimates on ${L}^p$ spaces for {S}chr{\"o}dinger operators with nonnegative potentials, Ann. Inst. Fourier 57 (2007) 1975--2013.
\newblock \href {https://doi.org/10.5802/aif.2320} {\path{doi:10.5802/aif.2320}}.

\bibitem{she}
Z.~Shen, ${L}^p$ estimates for {S}chr{\"o}dinger operators with certain potentials, Annales de l'Institut Fourier 45 (1995) 513--546.
\newblock \href {https://doi.org/10.5802/aif.1463} {\path{doi:10.5802/aif.1463}}.

\bibitem{can-ves}
P.~Cannarsa, V.~Vespri, Generation of analytic semigroups in the ${L}^p$ topology by elliptic operators in $\mathbb{R}^n$, Israel J. Math. 61~(3) (1988) 235–255.

\bibitem{tan}
H.~Tanabe, Functional analytic methods for partial differential equations, M. Dekker, 1997.

\bibitem{cup-for}
G.~Cupini, S.~Fornaro, Maximal regularity in ${L}^p(\mathbb{R}^{N})$ for a class of elliptic operators with unbounded coefficients, Diff. Int. Eqs. 17 (2004) 259--296.

\bibitem{lun}
A.~Lunardi, Analytic semigroups and optimal regularity in parabolic problems, Birkhauser, 1995.

\bibitem{agm62}
S.~Agmon, On the eigenfunctions and on the eigenvalues of general elliptic boundary value problems., Comm. Pure Appl. Math. 15 (1962) 119–147.

\bibitem{eng-nag}
K.~Engel, R.~Nagel, One parameter semigroups for linear evolutions equations, Springer-Verlag, Berlin, 2000.

\bibitem{bro-opi}
R.~C. Brown, B.~Opic, {E}mbeddings of weighted {S}obolev spaces into spaces of continuous functions, Proc. Roy. Soc. London Ser. A 439 (1992) 279–296.

\end{thebibliography}

\end{document}